\documentclass[12 pt]{amsart}
\usepackage{amsmath,amssymb,amsthm}
\usepackage{longtable}
\usepackage{a4wide}
\usepackage{color}
\usepackage{cleveref}
\usepackage{graphicx}

\newcommand{\Tra}{\mathrm{Tra}}
\newcommand{\Alt}{\mathrm{Alt}}
\newcommand{\Sym}{\mathrm{Sym}}
\newcommand{\Aut}{\mathrm{Aut}}
\newcommand{\Cay}{\mathrm{Cay}}
\newcommand{\K}{\mathrm{K}}
\newcommand{\A}{\mathrm{A}}
\newcommand{\V}{\mathrm{V}}
\newcommand{\ZZ}{\mathbb{Z}}
\newcommand{\NN}{{\mathbb N}}

\renewcommand{\geq}{\geqslant}
\renewcommand{\leq}{\leqslant}
\renewcommand{\ge}{\geqslant}
\renewcommand{\le}{\leqslant}

\newtheorem{defn}{Definition}[section]
\newtheorem{theorem}[defn]{Theorem}
\newtheorem{remark}[defn]{Remark}
\newtheorem{lemma}[defn]{Lemma}
\newtheorem{cor}[defn]{Corollary}

\newtheorem{prop}[defn]{Proposition}

\title[Digraphs that are Cayley on two nonisomorphic groups]{Digraphs with small automorphism groups that are Cayley on two nonisomorphic groups}
\author{Luke Morgan}
\author{Joy Morris}
\author{Gabriel Verret}

\address{Luke Morgan\\
School of Mathematics and Statistics (M019)\\
University of Western Australia\\
Crawley, 6009\\
Australia} 
\email{luke.morgan@uwa.edu.au}

\address{Joy Morris\\
Department of Mathematics and Computer Science\\
University of Lethbridge\\
Lethbridge, AB T1K 3M4\\
Canada}
\email{joy.morris@uleth.ca}

\address{Gabriel Verret\\
Department of Mathematics, The University of Auckland\\
Private Bag 92019, Auckland 1142, New Zealand.}
\email{g.verret@auckland.ac.nz}	

\thanks{This research was supported by the Australian Research Council grant DE160100081 and by the Natural Science and Engineering Research Council of Canada. The first and last authors also thank the second author and the University of Lethbridge for hospitality.}


\begin{document}
\begin{abstract}
Let $\Gamma=\Cay(G,S)$ be a Cayley digraph on a group $G$ and let $A=\Aut(\Gamma)$. The \emph{Cayley index} of $\Gamma$ is $|A:G|$. It has previously been shown that, if  $p$ is a prime, $G$ is a
 cyclic  $p$-group and $A$ contains a noncyclic regular subgroup, then the Cayley index of $\Gamma$ is superexponential in $p$. 

We present evidence suggesting that cyclic groups are exceptional in this respect. Specifically, we establish the contrasting result that, if $p$ is an odd prime and $G$ is abelian but not cyclic, and has order a power of $p$ at least $p^3$, then there is a Cayley digraph $\Gamma$ on $G$ whose Cayley index is just $p$, and whose automorphism group contains a nonabelian regular subgroup.  
\end{abstract}

\maketitle

\section{Introduction}
Every digraph and group in this paper is finite.  A \emph{digraph} $\Gamma$ consists of a set of \emph{vertices} $\V(\Gamma)$ and a set of \emph{arcs} $\A(\Gamma)$, each arc being an ordered pair of  distinct vertices.  (Our digraphs do not have loops.) We say that $\Gamma$ is a \emph{graph} if, for every arc $(u,v)$  of $\Gamma$, $(v,u)$ is also an arc. Otherwise, $\Gamma$ is a \emph{proper} digraph.  

The \emph{automorphisms} of $\Gamma$ are the permutations of $\V(\Gamma)$ that preserve $\A(\Gamma)$. They form a group under composition, denoted $\Aut(\Gamma)$.
 
Let $G$ be a group  and let $S$ be a subset of $G$ that does not contain the identity. The \emph{Cayley digraph} on $G$   with connection set $S$ is $\Gamma=\Cay(G,S)$, the digraph with  vertex-set  $G$ and where  $(u,v) \in \A(\Gamma)$ whenever $vu^{-1}\in S$.  The index of $G$ in $\Aut(\Gamma)$ is called the \emph{Cayley index} of $\Gamma$.

It is well-known that a digraph  is a Cayley digraph on $G$ if and only if  its automorphism group contains the right regular representation of $G$. A digraph may have more than one regular subgroup in its automorphism group and hence more than one representation as a Cayley digraph. This  is an interesting situation that has been studied in \cite{BamGiu,Dragan,Royle,Spiga,Boris}, for example.

Let $p$ be a prime. Joseph~\cite{Joseph} proved that if $\Gamma$ has order $p^2$ and $\Aut(\Gamma)$ has two regular subgroups, one of which is cyclic and the other not, then $\Gamma$ has Cayley index at least $p^{p-1}$. The second author generalised this in \cite{Joy-noniso}, showing that if $p\geq 3$, $\Gamma$ has order $p^n$ and $\Aut(\Gamma)$ has two regular subgroups, one of which is cyclic and the other not, then $\Gamma$ has Cayley index  at least $p^{p(n-1)-1}$. A simpler proof of this was later published in~\cite{Alspach-Du}. Kov\'{a}cs and Servatius~\cite{Kovacs-Servatius} proved the analogous result when $p=2$.

The theme of the results above is that if $\Aut(\Gamma)$ has two regular subgroups, one of which is cyclic and the other not, then $\Gamma$ must have ``large'' Cayley index. (In fact, close examination of the proofs of the results above reveals that $\Aut(\Gamma)$ having two distinct regular subgroups, one of which is cyclic,  might suffice. We will not dwell on this point.)

The goal of this paper is to show that cyclic $p$-groups are exceptional with respect to this property, at least among abelian $p$-groups. More precisely, we prove the following.

\begin{theorem}\label{main:theo}
Let $p$ be an odd prime and let $G$ be an abelian $p$-group. If $G$ has order at least $p^3$ and is not cyclic, then there exists a  proper Cayley digraph on $G$ with Cayley index $p$ and whose automorphism group contains a nonabelian regular subgroup.
\end{theorem}

It would be interesting to generalise \Cref{main:theo} to nonabelian $p$-groups and to $2$-groups. More generally, we expect that ``most'' groups admit a Cayley digraph of ``small'' Cayley index such that the automorphism group of the digraph contains another (or even a nonisomorphic) regular subgroup. At the moment, we do not know how to approach this problem in general, or even what a sensible definition of ``small'' might be. (\Cref{lemma:propsubgroup} shows that the smallest  index of a  proper subgroup of  either  of the regular subgroups is a lower bound -- and hence that the Cayley index of $p$ in \Cref{main:theo} is best possible.) As an example, we prove the following.

\begin{prop}\label{prop:symn}
Let  $G$ be a group generated by an involution $x$ and an element $y$ of order $3$, and such that $ \ZZ_6 \ncong G \ncong \ZZ_3\wr \ZZ_2$. If $G$ has a subgroup $H$ of index $2$, then there is a Cayley digraph $\Gamma$  with Cayley index $2$ such that $\Aut(\Gamma)$ contains a regular subgroup  distinct from $G$ and isomorphic to $H\times\ZZ_2$.
\end{prop}

This paper is laid out as follows.  ~\Cref{cart-prod} includes structural results on cartesian products of digraphs  that will be required in the proofs of our main results, while in \Cref{prelim} we collect results about automorphism groups of digraphs. \Cref{sec:main} consists of the proof of~\Cref{main:theo}. Finally, in \Cref{sec:prop} we prove~\Cref{prop:symn} and consider the case of  symmetric groups.

\section{Cartesian products}\label{cart-prod}

The main result of this section is a version of  a result about cartesian products of graphs due  to Imrich~\cite[Theorem 1]{Imrich} that is adapted to the case of proper digraphs. Imrich's proof can be generalised directly to all digraphs, but his proof involves a detailed case-by-case analysis for small graphs,  which can be avoided by restricting attention to proper digraphs.

The \emph{complement} of a digraph $\Gamma$, denoted $\overline{\Gamma}$, is the digraph with vertex-set $\V(\Gamma)$, with  $(u,v) \in \A(\overline{\Gamma})$ if and only if $(u,v)\not\in \A(\Gamma)$, for every two  distinct  vertices $u$ and $v$ of $\Gamma$. It is easy to see that a digraph and its complement have the same automorphism group.

Given digraphs $\Gamma$ and $\Delta$, the \emph{cartesian product} $\Gamma \square \Delta$ is the digraph  with vertex-set $\V(\Gamma)\times \V(\Delta)$ and with $((u,v),(u',v'))$ being an arc if and only if either $u=u'$ and $(v,v')\in\A(\Delta)$, or $v=v'$ and $(u,u')\in\A(\Gamma)$. For each $u\in \V(\Gamma)$, we obtain a  \emph{copy}  $\Delta^u$ of $\Delta$ in $\Gamma \square \Delta$, the induced digraph on $\{(u,v) \mid v \in \V(\Delta)\}$. Similarly, for each $v\in \V(\Delta)$, we obtain a copy $\Gamma^v$ of $\Gamma$ in $\Gamma\square \Delta$ (defined analogously).

A digraph $\Gamma$ is \emph{prime} with respect to the cartesian product if the existence of an isomorphism from $\Gamma$ to $\Gamma_1 \square \Gamma_2$  implies that either $\Gamma_1$ or $\Gamma_2$ has order $1$, so that $\Gamma$ is isomorphic to either $\Gamma_1$ or $\Gamma_2$.

It is well known that, with respect to the cartesian product, graphs  can be factorised uniquely as a product of prime factors. Digraphs also have this property. In fact, the following stronger result holds.

\begin{theorem}[Walker, {\cite{Walker}}]\label{strict-factorisation}
Let $\Gamma_1,\ldots,\Gamma_k$, $\Gamma_1',\ldots,\Gamma_\ell'$ be prime digraphs. If $\alpha$ is an isomorphism   from $\Gamma_1 \square \cdots \square \Gamma_k$ to $\Gamma_1' \square\cdots\square\Gamma_k'$, then $k=\ell$ and there exist a permutation $\pi$ of $\{1, \ldots, k\}$ and isomorphisms $\alpha_i$ from $\Gamma_i$ to $\Gamma'_{\pi(i)}$ such that $\alpha$ is the product of the $\alpha_i$s ($1 \le i \le k$). 
\end{theorem}

\Cref{strict-factorisation} is a corollary of \cite[Theorem 10]{Walker}, as noted in the ``Applications" section of \cite{Walker}. We now present the version of Imrich's result that applies to proper digraphs.

\begin{theorem}\label{cartesian}
If $\Gamma$ is a proper digraph, then at least one of $\Gamma$ or $\overline{\Gamma}$ is prime with respect to cartesian product.
\end{theorem}

\begin{proof} 
Towards a contradiction, assume that $\Gamma=\overline{A\square B}$ and that $\varphi$ is an isomorphism from $\Gamma$ to $C\square D$, where $A$, $B$, $C$, and $D$ all have at least 2 vertices.  A key observation is the fact that if  $x$ and $y$ are distinct vertices in the same copy of $X$ in a cartesian product $X \square Y$, then every vertex contained in at least one arc with each of $x$ and $y$ must also lie in that copy of $X$.

Since $\Gamma$ is a proper digraph, without loss of generality so is $A$, and $A$ has an arc $(a,a')$ such that $(a',a)$ is not an arc of $A$. Let $b$ be a vertex of $B$. 

Pick $b'$ to be a  vertex of $B$ distinct from $b$. We claim that $\varphi( (a,b))$, $\varphi((a',b))$, $\varphi((a,b'))$, $\varphi((a',b'))$ all lie in some copy of either $C$ or $D$.
The digraph below is the subdigraph of $\Gamma$ under consideration.

\begin{center}\includegraphics[scale=0.8]{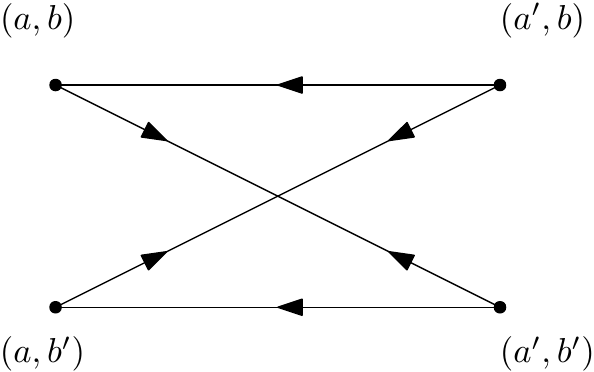}\end{center}

Since every arc in $C \square D$ lies in either a copy of $C$ or $D$, we may assume that the arc from $\varphi((a,b))$ to $\varphi((a',b'))$ lies in some copy $C^d$ of $C$, say $\varphi((a,b))=(c,d)$ and $\varphi((a',b'))=(c',d)$, with $c,c'\in\V(C)$ and $d\in\V(D)$. Towards a contradiction, suppose that  $\varphi((a',b)) \notin C^d$. Then the arc from $\varphi((a',b))$ to $(c,d)$ must lie in $D^c$, so $\varphi((a',b))=(c,d')$ for some vertex $d'$ of $D$. Since there is a path of length $2$ via $\varphi((a,b'))$ from  $(c',d)$  to $(c,d')$ and since $\varphi((a,b')) \neq (c,d)=\varphi((a,b))$, we must have $\varphi((a,b'))=(c',d')$. But now we have an arc from $(c,d')$ to $(c,d)$  and an arc from $(c',d)$ to $(c',d')$, so arcs in both directions between $d$ and $d'$ in $D$. This implies that there are arcs in both directions between $(c,d)=\varphi((a,b))$ and  $(c,d')=\varphi((a',b))$, a contradiction. Hence $\varphi((a',b)) \in C^d$, and by the observation in the first paragraph, we have $\varphi((a,b')) \in C^d$ also. This proves the claim. 

By repeatedly applying the claim, all elements of   $\varphi(\{a,a'\}\times \V(B))$ lie in some copy of $C$ or $D$, say, $C^d$. Let $a''\in\V(A)-\{a,a'\}$ and let $b$ and $b'$ be distinct vertices of $B$. By the definitions of cartesian product and complement, there are arcs in both directions between $(a'',b)$ and $(a,b')$ and between $(a'',b)$ and $(a',b')$. Thus, by the observation in the first paragraph, $\varphi((a'',b))$ also lies in $C^d$. This shows that every vertex of $\Gamma$ lies in $C^d$, so $D$ is trivial. This is the desired contradiction.
\end{proof}

\begin{remark}
Imrich's Theorem~\cite[Theorem 1]{Imrich} states that, for every graph $\Gamma$, either $\Gamma$ or $\overline{\Gamma}$ is prime with respect to the cartesian product, with the following exceptions :  $\K_2 \square \K_2$, $\K_2 \square \overline{\K_2}$,  $\K_2 \square \K_2 \square \K_2$, $\K_4 \square \K_2$, $\K_2 \square \K_4^-$, and $\K_3 \square \K_3$, where $\K_n$ denotes the complete graph on $n$ vertices and $\K_4^-$ denotes $\K_4$ with  an edge deleted. These would therefore be the complete list of  exceptions to \Cref{cartesian} if we removed the word `proper' from  the hypothesis.
\end{remark}

\begin{remark}
While most of our results apply only to finite digraphs,~\Cref{cartesian} also applies to infinite ones (as does Imrich's Theorem). The proof is the same.
\end{remark}

\begin{cor}\label{thm:non-exploder}
Let $\Gamma_1$ be a proper Cayley digraph on $G$ with Cayley index $i_1$ and let $\Gamma_2$ be a Cayley graph on $H$ with Cayley index $i_2$. If $i_1>i_2$, then at least one of $\Gamma_1\square \Gamma_2$ or $\overline{\Gamma_1} \square\Gamma_2$ has automorphism group equal to $\Aut(\Gamma_1)\times \Aut(\Gamma_2)$ and, in particular, is a proper Cayley digraph on $G\times H$ with Cayley index $i_1i_2$.
\end{cor}

\begin{proof}
 By~\Cref{cartesian}, one of $\Gamma_1$ and $\overline{\Gamma_1}$ is prime with respect to the cartesian  product, say $\Gamma_1$ without loss of generality.  Clearly, we have  $\Aut(\Gamma_1)\times \Aut(\Gamma_2)\leq \Aut(\Gamma_1\square \Gamma_2)$. Since $i_1>i_2$, $\Gamma_1$ cannot be a cartesian factor of $\Gamma_2$.  It follows by~\Cref{strict-factorisation} that every automorphism of $\Gamma_1 \square \Gamma_2$ is a product of an automorphism of $\Gamma_1$ and an automorphism of $\Gamma_2$, so that  $\Aut(\Gamma_1)\times \Aut(\Gamma_2)=\Aut(\Gamma_1\square \Gamma_2)$.
\end{proof}

\section{ Additional background}\label{prelim}

The following lemma is well known and easy to prove.

\begin{lemma}\label{lemma:aut}
Let $G$ be a group, let $S\subseteq G$ and let $\alpha\in\Aut(G)$. If $S^\alpha=S$, then $\alpha$ induces an automorphism of $\Cay(G,S)$ which fixes the vertex corresponding to the identity.
\end{lemma}

The next lemma is not used in any of our proofs, but it shows that the Cayley indices in~\Cref{main:theo} and~\Cref{prop:symn} are as small as possible.

\begin{lemma}\label{lemma:propsubgroup}
If $\Cay(G,S)$ has Cayley index $i$ and $\Aut(\Cay(G,S))$ has at least two regular subgroups, then $G$ has a proper subgroup of index at most $i$.
\end{lemma}
\begin{proof}
Let $A=\Aut(\Cay(G,S))$ and let $H$ be a regular subgroup of $A$ different from $G$. Clearly, $G\cap H$ is a proper subgroup of $G$ and we have $|A|\geq |GH|=\frac{|G||H|}{|G\cap H|}$ hence $i=|A:G|=|A:H|\geq |G:G\cap H|$.
\end{proof}

Generally, there are two notions of connectedness for digraphs: a digraph is \emph{weakly connected} if its underlying graph is connected, and \emph{strongly connected} if for every ordered pair of vertices there is a directed path from the first to the second. In a finite Cayley digraph, these notions coincide (see \cite[Lemma 2.6.1]{godsilroyle} for example). For this reason, we refer to Cayley digraphs as simply  being \emph{connected} or \emph{disconnected}.

If $v$ is vertex of a digraph $\Gamma$, then $\Gamma^+(v)$ denotes the {\emph{outneighbourhood}} of $v$, that is, the set of vertices $w$ of $\Gamma$ such that $(v,w)$ is an arc of $\Gamma$.

Let $A$ be a group of automorphisms of a digraph $\Gamma$.  For $v\in \V(\Gamma)$  and $i\geq 1$, we use $A_v^{+[i]}$ to denote the subgroup of $A_v$ that fixes every vertex $u$ for which there is a directed path of length at most $i$ from $v$ to $u$.

\begin{lemma}\label{lemma:fixNeighbourhood}
Let $\Gamma$ be a connected digraph, let $v$ be a vertex of $\Gamma$ and let $A$ be a transitive group of automorphisms of $\Gamma$. If $A_v^{+[1]}=A_v^{+[2]}$, then $A_v^{+[1]}=1$. 
\end{lemma}

\begin{proof}
By the transitivity of $A$, we have $A_u^{+[1]}=A_u^{+[2]}$ for every vertex $u$. Using induction on $i$, it easily follows that, for every $i\geq 1$, we have $A_v^{+[i]}=A_v^{+[i+1]}$. By connectedness, this implies that  $A_v^{+[1]}=1$.
\end{proof}

\begin{lemma}\label{lemma:nonabelian}
Let $p$ be a prime and let $A$ be a permutation group whose order is a power of $p$.  If $A$ has a regular abelian subgroup $G$ of index $p$ and $G$ has a subgroup $M$ of index $p$  that is normalised but not centralised by  a point-stabiliser  in $A$, then $A$ has a regular nonabelian subgroup. 
\end{lemma}
\begin{proof}
Let $A_v$ be a point-stabiliser in $A$. Note that $A=G \rtimes A_v$ and that $|A_v|=p$.   Since $M$ is normal in $G$ and normalised by $A_v$, it is normal in $A$ and has index $p^2$. Clearly, $M\rtimes A_v\neq G$ hence $A/M$ contains at least two subgroups of order $p$ and must therefore be elementary abelian.

Let $\alpha$ be a generator of $A_v$ and let $g\in G-M$. By the previous paragraph, we have  $(g\alpha)^p \in M$.  Let  $H=\langle M,g\alpha\rangle$. Since $M$ is centralised by $g$ but not by $\alpha$, it is not centralised by $g\alpha$ hence $H$ is nonabelian. Further, we have $|H|=p|M| = |G|$, so that $H$ is normal in $A$. If  $H$ was non-regular, it would contain all point-stabilisers of $A$, and thus would contain $\alpha$ and hence also $g$. This would  give $G=\langle M, g \rangle \leqslant H$, a contradiction. Thus $H$ is a regular nonabelian subgroup of $A$.
\end{proof}

\section{Proof of \Cref{main:theo}}\label{sec:main}

Throughout this section, $p$ denotes an odd prime. In~\Cref{sec:Z_p^3}, we show that \Cref{main:theo} holds when $G\cong \ZZ_p^3$. In Sections~\ref{sec:Z_p^n x Z_p} and~\ref{sec:rk 2a}, we subdivide abelian groups of rank $2$ and order at least $p^3$ into two families,  and show that the theorem holds for all such groups. Finally, in~\Cref{sec:mainpf}, we explain how these results can be applied to show that the theorem holds for all abelian groups of order at least $p^3$.

\subsection{\mathversion{bold}$G\cong\ZZ_p^3$}\label{sec:Z_p^3}\label{sec:z_p^3}
Write $G=\langle x,y,z\rangle$, let $\alpha$ be the automorphism of $G$ that maps $(x,y,z)$ to $(xy,yz,z)$, let $S=\{x^{\alpha^i},y^{\alpha^i}:i \in \ZZ\}$, let $\Gamma=\Cay(G,S)$, and let $A=\Aut(\Gamma)$. Note that $\Gamma$ is a proper digraph (this will be needed in \Cref{sec:mainpf}).

It is easy to see that, for $i\in\NN$, we have $x^{\alpha^i}=xy^iz^{\binom{i}{2}}$, $y^{\alpha^i}=yz^i$ and $z^{\alpha^i}=z$. In particular, $\alpha$ has order $p$ and $|S|=2p$. By \Cref{lemma:aut}, $G\rtimes \langle \alpha\rangle\leq A$. We will show that equality holds.

Using the formulas above, it is not hard to see that  the induced digraph on $S$ has exactly $2p$ arcs:  $(x^{\alpha^i},x^{\alpha^{i+1}})$ and $(y^{\alpha^i},x^{\alpha^{i+1}})$, where $i\in \ZZ_p$.  Thus,
 for every $s \in S$, $A_{1,s}=A_1^{+[1]}$. By vertex-transitivity, $A_{u,v}=A_u^{+[1]}$ for every arc $(u,v)$.

Let $s\in S$.  We have already seen that $A_{1,s}=A_1^{+[1]}$. Let $t\in S$. From the structure of the induced digraph on $S=\Gamma^{+}(1)$, we see that $t$ has an out-neighbour in $S$, so that both $t$ and this out-neighbour are fixed by $A_{1,s}$. It follows that $A_{1,s}$ fixes all out-neighbours of $t$. We have shown that $A_{1,s}=A_1^{+[2]}$.  By \Cref{lemma:fixNeighbourhood}, it follows that $A_{1,s}=1$. Since   the induced digraph on $S$ is not vertex-transitive and  $\alpha \in A_1$,  the $A_1$-orbits on $S$ have length $p$. Hence $|A_1|=p|A_{1,s}|=p$. Thus, $\Gamma$ has Cayley index $p$ and $A=G\rtimes \langle \alpha\rangle$. Finally, we apply~\Cref{lemma:nonabelian} with $M=\langle y,z\rangle$ to deduce that $A$ contains a nonabelian regular subgroup.

\subsection{\mathversion{bold}$G\cong\ZZ_{p^n} \times \ZZ_{p}$ with $n \geqslant 2$}\label{sec:Z_p^n x Z_p}\label{sec:rk 2 exp>p}

Write $G=\langle x,y\rangle$, let $x_0=x^{p^{n-1}}$, let $\alpha$ be the  automorphism of $G$ that maps $(x,y)$ to $(xy,x_0y)$, let $S=\{x^{\alpha^i},y^{\alpha^i} :i \in \ZZ\}$ and let $\Gamma=\Cay(G,S)$. Again, note that $\Gamma$ is a proper digraph.

Since $n\geq 2$, $x_0$ is fixed by $\alpha$. It follows that, for $i\in\NN$, we have $x^{\alpha^i}=xy^ix_0^{\binom{i}{2}}$ and $y^{\alpha^i}=yx_0^i$. In particular, $\alpha$ has order $p$ and $|S|=2p$. 

Using these formulas, it is not hard to see that the induced digraph on $S$ has exactly $2p$ arcs:  $(x^{\alpha^i},x^{\alpha^{i+1}})$ and $(y^{\alpha^i},x^{\alpha^{i+1}})$, where $i\in \ZZ_p$. The proof is now exactly as in the previous section, except that we use $M=\langle x^p,y\rangle$ when applying~\Cref{lemma:nonabelian}.

\subsection{\mathversion{bold}$G\cong\ZZ_{p^n} \times \ZZ_{p^m}$ with $n\geqslant m \geqslant 2$}\label{sec:rk 2a}

Write $G=\langle x,y\rangle$, let $x_0=x^{p^{n-1}}$, let $y_0=y^{p^{m-1}}$,  let $\alpha$ be the  automorphism of $G$ that maps $(x,y)$ to $(xy_0,yx_0)$, let $S=\{x^{\alpha^i},y^{\alpha^i}, (xy^{-1})^{\alpha^i}:i \in \ZZ\}$, let $\Gamma=\Cay(G,S)$, and let $A=\Aut(\Gamma)$. Again, note that $\Gamma$ is a proper digraph.

Since $n\geqslant m \geqslant 2$,  $x_0$ and $y_0$ are both fixed by $\alpha$. It follows that, for $i\in\NN$, we have $x^{\alpha^i}=xy_0^i$, and $y^{\alpha^i}=yx_0^i$. In particular, $\alpha$ has order $p$ and $|S|=3p$. By \Cref{lemma:aut}, $G\rtimes \langle \alpha\rangle\leq A$.

Using the formulas above, it is not hard to see that the induced digraph on $S$ has exactly $2p$ arcs:  $((xy^{-1})^{\alpha^i},x^{\alpha^i})$ and $(y^{\alpha^i},x^{\alpha^{i}})$, where $i\in \ZZ_p$. It follows that $|A_1:A_{1,x}|=p$. We will show that $A_{1,x}=1$, which will imply that $A=G\rtimes \langle \alpha\rangle$.

Let $X=\{x^{\alpha^i}:i\in\ZZ\}=x\langle y_0\rangle$, $Y=\{x^{\alpha^i}:i\in\ZZ\}=y \langle x_0\rangle$ and $Z=\{(xy^{-1})^{\alpha^i}:i\in\ZZ\}=xy^{-1}\langle x_0^{-1}y_0\rangle$. It follows from the previous paragraph that $X$ is an orbit of $A_1$ on $S$. 

Note that the $p$ elements of $Y^2=y^2\langle x_0\rangle$ are out-neighbours of every element of $Y$. Similarly, the $p$ elements of $Z^2$ are out-neighbours of every element of $Z$. On the other hand, one can check that an element of $Y$ and an element of $Z$ have a unique out-neighbour in common, namely their product. This shows that $Y$ and $Z$ are blocks for $A_1$. We claim that $Y$ and $Z$ are orbits of $A_1$.

Let $Y_1=Y$ and, for $i\geq 2$, inductively define $Y_i= \bigcap_{x\in Y_{i-1}} \Gamma^+(x)$. Define $Z_i$ analogously. Let $g\in A_1$. By induction, $Y_i^g \in \{Y_i, Z_i\}$, and $Y_i^g = Z_i$  if and only if $Y^g = Z$. Note that $Y_i = Y^i=y^i \langle x_0 \rangle$ and $Z_i = Z^i=x^iy^{-i}\langle x_0^{-1}y_0 \rangle$.  If $n=m$, then $1\in x_0y_0^{-1}\langle x_0^{-1}y_0\rangle=Z_{p^{m-1}}$, but $1\notin y_0\langle x_0\rangle=Y_{p^{m-1}}$, so $Y$ and $Z$ are orbits for $A_1$. We may thus assume that $n>m$. Note that $y_0\in y_0 \langle x_0\rangle=Y_{p^{m-1}}$, and $y_0$ is an in-neighbour of $x\in X$. However, $Z_{p^{m-1}}=x^{p^{m-1}}y_0^{-1} \langle x_0^{-1}y_0\rangle$.  Since $n>m$, we see that no vertex of $Z_{p^{m-1}}$ is an in-neighbour of a vertex of $X$.  Again it follows that $Y$ and $Z$ are orbits for $A_1$.

Considering the structure of the induced  digraph on $S$, it follows that, for every $s \in S$, $A_{1,s}=A_1^{+[1]}$. By vertex-transitivity, $A_{u,v}=A_u^{+[1]}$ for every arc $(u,v)$.  Since elements of $Y$ and $Z$ have an out-neighbour in $S$, $A_{1,x}$ fixes the out-neighbours of elements of $Y$ and $Z$. Furthermore, for every $i\in\ZZ$,  $xy_0^iy$ is a common outneighbour of $xy_0^i$ and $y$, hence it is fixed by $A_{1,x}$. Thus, every element of $X$ has an out-neighbour fixed by $A_{1,x}$. It follows that $A_{1,x}$ fixes all out-neighbours of elements of $X$ and thus $A_{1,x}=A_1^{+[1]}=A_1^{+[2]}$. By \Cref{lemma:fixNeighbourhood}, it follows that $A_{1,x}=1$. As in~\Cref{sec:Z_p^3}, we can also conclude $|A_1|=p$, $\Gamma$ has Cayley index $p$ and $A=G\rtimes \langle \alpha\rangle$. Finally, applying \Cref{lemma:nonabelian} with $M=\langle x^p,y\rangle$ implies that $A$ contains a nonabelian regular subgroup.

\subsection{General case}\label{sec:mainpf}
Recall that $G$ is an abelian $p$-group that has order at least $p^3$ and is not cyclic. By the Fundamental Theorem of Finite Abelian Groups, we can write $G=G_1\times G_2$, where $G_1$ falls into one of the three cases that have already been dealt with in this section.

(More explicitly, if $G$ is not elementary abelian, then we can take $G_1$ isomorphic to $\ZZ_{p^n}\times \ZZ_{p^m}$ with $n \ge 2$ and $m \ge 1$. If $G$ is elementary abelian, then, since $|G|\geq p^3$, we can take $G_1$ isomorphic to $\ZZ_p^3$.)

We showed in the previous three sections that there exists a proper Cayley digraph $\Gamma_1$ on $G_1$ with Cayley index equal to $p$ and whose automorphism group contains a nonabelian regular subgroup.

Note that every cyclic group admits a Cayley digraph whose Cayley index is  $1$. (For example, the directed cycle of the corresponding order.) Since $G_2$ is a direct product of cyclic groups, applying~\Cref{thm:non-exploder} iteratively yields a proper Cayley digraph $\Gamma$ on $G_1\times G_2$ with automorphism group $\Aut(\Gamma_1)\times G_2$. In particular, $\Gamma$ has Cayley index $p$ and its automorphism group contains a nonabelian regular subgroup. This concludes the proof of \Cref{main:theo}.

In fact, the proof above yields the following stronger result.

\begin{theorem}
Let $G$ be an abelian group. If there is an odd prime $p$ such that the Sylow $p$-subgroup of $G$ is neither cyclic nor elementary abelian of rank $2$, then $G$ admits a proper Cayley digraph with Cayley index $p$ whose automorphism group contains a nonabelian regular subgroup. 
\end{theorem}

\section{Proof of \Cref{prop:symn}}\label{sec:prop}

We begin with a lemma that helps to establish the existence of regular subgroups.

\begin{lemma}\label{second-subgroup}
 Let $G$ be a group with nontrivial subgroups $H$ and $B$ such that $G=HB$ and $H \cap B=1$, and  let $\Gamma=\Cay(G,S)$ be a Cayley digraph on $G$. If $S$ is closed under conjugation by $B$, then $\Aut(\Gamma)$ has a regular subgroup distinct from the right regular representation of $G$ and isomorphic to $H \times B$.
\end{lemma}

\begin{proof} 
Let $A=\Aut(\Gamma)$. For $g\in G$, let $\ell_g$ and $r_g$  denote the permutations of $G$ induced by left and  right multiplication by $g$,  respectively. Similarly, for $g\in G$, let $c_g$ denote the permutation of $G$ induced by conjugation by $g$. For  $X \leqslant G$, let $R_X=\langle r_x: x \in X \rangle$. Let  $L_B=\langle \ell_b: b \in B\rangle $ and  $C_B=\langle c_b: b \in B \rangle$. Note that $R_H \leqslant A$. For  every $g \in G$, $r_gc_{g^{-1}}=\ell_g$. For all $b\in B$, we have $r_b \in A$ and, since $S$ is closed under conjugation by $B$,  $c_{b^{-1}} \in A$ hence   $L_B\leq A$. 

Let $K = \langle  L_B, R_H\rangle$.   If $K=R_G$, then $L_B\leq R_G$ which implies that $C_B\leq R_G$, contradicting the fact that $R_G$ is regular. Thus $K\neq R_G$. Note that  $L_B$ and  $R_H$  commute. Suppose that $k \in R_H \cap L_B$, so $k=r_h=\ell_b$ for some $h \in H$ and some $b\in B$. Thus 
$$h = 1^{r_h}=1^k = 1^{\ell_b} = b.$$

Since $H\cap B=1$, this implies $k=1$. It follows that $R_H \cap L_B=1$ and hence $K=R_H\times L_B \cong H\times B$. Finally, suppose that some $k=r_h \ell_b\in K$ fixes $1$. It follows that $1^{r_h \ell_b}=1=bh$ so that $b\in H$, a contradiction. This implies that $K$ is regular, which concludes the proof.
\end{proof}

We now prove a general result, which together with~\Cref{second-subgroup} will imply~\Cref{prop:symn}.

\begin{prop}\label{Order2Order3}
Let $G$ be a group generated by an involution $x$ and an element $y$ of order $3$, let $S=\{ x, y ,y^x \}$ and let $\Gamma=\Cay(G,S)$. If $G$ is isomorphic to neither $\ZZ_6$ nor $\ZZ_3\wr\ZZ_2\cong\ZZ_3^2\rtimes\ZZ_2$, then $\Gamma$ has Cayley index $2$. 
\end{prop}

\begin{proof}
Clearly, $\Gamma$ is connected.
Since $G$ is not isomorphic to $\ZZ_6$, we have $y^x\neq y$. In particular, we have $|S|=3$. If $y^x=y^{-1}$, then $G\cong\Sym(3)$ and the result can be checked directly. We therefore assume that $y^x\neq y^{-1}$. Since $G\ncong \ZZ_3\wr\ZZ_2$, we have $y^xy\neq yy^x$. 

We have that $\Gamma^+(x)=\{1,yx,xy\}$, $\Gamma^+(y)=\{xy,y^2, y^xy\}$ and $\Gamma^+(y^x)=\{yx,yy^x,(y^2)^x\}$. One can check that the only equalities between elements of these sets are the ones between elements having the same representation. In other words,  $|\{1,yx,xy,y^2,y^xy,yy^x,(y^2)^x\}|=7$. (For example, if $yx=y^xy$, then $x^{y^{-1}}=y^x$, contradicting the fact that $x$ and $y$ have different orders.)

Let $A=\Aut(\Gamma)$ and let $c_x$ denote conjugation by $x$. Note that $c_x\in A_1$. We first show that $A_1^{+[1]}=1$.  It can be checked that $y^2$ is the unique out-neighbour of $y$ that is also an in-neighbour of $1$, hence it is fixed by $A_1^{+[1]}$, and so is $(y^2)^x$ by analogous reasoning. We have seen earlier that $xy$ is the unique common out-neighbour of $x$ and $y$, hence it too is fixed by $A_1^{+[1]}$, and similarly for $yx$. Being the only remaining out-neighbours of $y$, $y^xy$ must be also fixed, and similarly for $yy^x$. Thus $A_1^{+[1]}=A_1^{+[2]}$. Since $\Gamma$ is connected,~\Cref{lemma:fixNeighbourhood} implies that $A_1^{+[1]}=1$.

Note that $x$ is the only out-neighbour of $1$ that is also an in-neighbour, hence it is fixed by $A_1$, whereas $c_x$ interchanges $y$ and $y^x$. It follows that $|A_1|=|A_1:A_1^{+[1]}|=2$ and $\Gamma$ has Cayley index $2$, as desired.
\end{proof}

\begin{proof}[Proof of~\Cref{prop:symn}]
Let $\Gamma=\Cay(G,\{x,y,y^x\})$. By~\Cref{Order2Order3}, $\Gamma$ has Cayley index $2$.  Since $|G:H|=2$ and $y$ has order $3$,  we have $y \in H$. As $\langle x,y\rangle=G$, we have $x \notin H$ and $G=H\rtimes\langle x\rangle$.   Clearly, $\{x,y,y^x\}$ is closed under conjugation by $x$. It follows by~\Cref{second-subgroup} that  $\Aut(\Gamma)$ has a regular subgroup distinct from $G$ and isomorphic to $H \times \langle x\rangle \cong H \times \ZZ_2$.
\end{proof}

It was shown by Miller  \cite{miller} that, when  $n \ge 9$, $\Sym(n)$ admits a generating set consisting of an element  of order $2$ and one of order $3$; this is also true when $n\in\{3,4\}$. In these cases, we can apply~\Cref{prop:symn} with $H= \Alt(n)$ to obtain a Cayley digraph  on $\Sym(n)$ that has Cayley index $2$ and whose automorphism group contains a regular subgroup isomorphic to $\Alt(n) \times \ZZ_2$.

A short alternate proof of this fact can be derived from a result of Feng~\cite{feng}. This yields a Cayley graph and is valid for $n\geq 5$.

\begin{prop}\label{newprop}
 If $n \ge 5$, then there is a Cayley graph on $\Sym(n)$ with Cayley index $2$, whose automorphism group contains a regular subgroup isomorphic to $\Alt(n) \times \ZZ_2$.
\end{prop}

\begin{proof}
Let $T=\{(1\ 2),(2\ 3), (2\ 4)\} \cup \{(i\ i+1): 4 \le i \le n-1\}$  and let $\Gamma=\Cay(\Sym(n),T)$. Note that all elements of $T$ are transpositions. Let $\Tra(T)$ be the transposition graph of $T$, that is,  the graph with vertex-set  $\{1, \ldots, n\}$ and with  an edge $\{i,j\}$ if and only if $(i\ j) \in T$. Note that $\Tra(T)$ is a tree and thus $T$ is a minimal generating set for $\Sym(n)$ (see for example~\cite[Section~3.10]{godsilroyle}).  Let $B=\langle (1\ 3)\rangle$. Since $n \ge 5$, $\Aut(\Tra(T))=B$. It follows by~\cite[Theorem 2.1]{feng} that $\Aut(\Gamma)  \cong \Sym(n) \rtimes B$. In particular, $\Gamma$  has Cayley index $2$.

Note that $\Sym(n)=\Alt(n) \rtimes B $ and that $T$ is closed under conjugation by $B$. Applying ~\Cref{second-subgroup} with $H=\Alt(n)$  shows that $\Aut(\Gamma)$ has a regular subgroup isomorphic to $\Alt(n) \times \ZZ_2$.
\end{proof}

\end{document}